\documentclass[12pt]{amsart}
\usepackage{graphicx}
\usepackage{amsmath}
\usepackage{amssymb}
\usepackage{amsthm}
\usepackage{esint}
\usepackage{xcolor}

\usepackage[top=15mm, bottom=30mm, left=25mm, right=25mm]{geometry}

\newtheorem{theorem}{Theorem}[section]

\theoremstyle{Corollary}
\newtheorem{cor}[theorem]{Corollary}
\newtheorem{prop}[theorem]{Proposition}
\newtheorem{rema}[theorem]{Remark}
\newtheorem{prob}[theorem]{Problem}

\numberwithin{equation}{section}

\begin{document}

\title{Notes on the uniqueness of Type II Yamabe metrics}

\author{Shota Hamanaka}

\address{Department of Mathematics, Graduate School of Science, Osaka University, Toyonaka, Osaka 560-0043, Japan}

\email{hamanaka.shota.sci@osaka-u.ac.jp}

\author{Pak Tung Ho}
\address{Department of Mathematics, Tamkang University, Tamsui, New Taipei City 251301, Taiwan}

\email{paktungho@yahoo.com.hk}

\subjclass[2020]{Primary; 53C18 Secondary; 53C21, 58J32}

\date{20th of September, 2024}

\begin{abstract}
In this paper, we study the uniqueness of type II Yamabe metrics in conformal classes on a compact connected manifold with boundary, and
we investigate Obata-type theorems for type II Yamabe metrics.
In particular, we establish a theorem which gives a sufficient condition for a metric to be the unique Type II Yamabe metric in its conformal class.
We also prove the corresponding theorem for the CR Yamabe problem on closed manifolds.

\color{blue}\textbf{This preprint has not any post-submission improvements or corrections. The Version of Record of this article is 
published in \textit{Nonlinear Differential Equations and Applications NoDEA}, and is available online at https://doi.org/10.1007/s00030-025-01092-0}\color{black}.
\end{abstract}

\maketitle

\section{Introduction}

The classical Yamabe problem on closed (i.e. compact without boundary) manifolds was completely solved by Yamabe \cite{yamabe1960deformation}, Trudinger \cite{trudinger1968remarks}, Aubin \cite{aubin1976equations} and Schoen \cite{schoen1984conformal}.
Escobar \cite{escobar1992yamabe, escobar1992conformal} introduced and studied two corresponding problems on compact manifolds with boundary.
The first problem \cite{escobar1992yamabe} is that for a given conformal class $C$ on a compact connected manifold with boundary, find a metric $g \in C$ with constant scalar curvature and minimal boundary (i.e., the mean curvature is zero along the boundary).
The second one \cite{escobar1992conformal} is that for a given conformal class $C$ on a compact connected manifold with boundary, find a scalar-flat metric $g \in C$ with constant mean curvature
on the boundary.
Similar to the studies of the classical Yamabe problem, Escobar \cite{escobar1992conformal, escobar1992yamabe} carried out the variational method to prove the existence of such metrics.
In that process, he defined certain conformal invariant corresponding to each problem, which was defined as the infimum of the appropriate energy in a fixed conformal class.
We call it corresponding to the second problem as \textit{the type II Yamabe constant} and a metric attaining this value as \textit{the type II Yamabe metric}
(see Section \ref{section-2} for the precise definition).

On closed manifolds, every metric with constant scalar curvature is a Yamabe metric in its conformal class up to a rescaling provided that its scalar curvature is nonpositive. This fact follows from the maximum principle.
Similarly, on compact manifolds with boundary, every metric with constant scalar curvature with minimal boundary is a relative Yamabe metric (i.e., the Yamabe metric corresponding to Escobar's first problem) in its relative conformal class (i.e., the set of conformally related metrics with minimal boundary) up to a rescaling provided that its scalar curvature is nonpositive.
And, on compact manifolds with boundary, every metric with zero scalar curvature and constant mean curvature on boundary is a type II Yamabe metric in its conformal class up to a rescaling provided that its mean curvature is nonpositive.

On the other hand, on closed manifolds, Obata's theorem \cite{obata1962certain, obata1971conjectures} states that a metric with constant scalar curvature is unique in its conformal class (up to a conformal diffeomorphism and a rescaling) if the conformal class contains an Einstein metric.
In particular, an Einstein metric is a Yamabe metric (see Section \ref{section-obata} of this paper).
However the converse is not true in general. Indeed, there are examples of non-Einstein but Yamabe metrics which were given by Kato in \cite{shin1994examples}.
On compact manifolds with boundary, there are some kinds of Obata-type theorems for Escobar's first problem \cite{akutagawa2021obata, escobar1990uniqueness}.
In \cite{hamanaka2021non}, the first author gave examples of non-Einstein but Yamabe metrics with minimal boundaries.
See also \cite{akutagawa2021obata} for another example.

In this paper, we give a sufficient condition for a metric to be a (unique) type II Yamabe metric in its conformal class on a compact connected manifold with boundary (see Theorem \ref{thm1}), and investigate Obata-type theorems for type II Yamabe metrics (see Section \ref{section-obata} and \ref{section-nonumbilic}).
We also prove the corresponding results in the CR case (see Section \ref{section-cr}).

\section{Type II Yamabe metrics}\label{section-2}

Let $M$ be a compact, connected, $n$-dimensional smooth manifold
with boundary $\partial M$, where $n\geq 3$.
Let $g$ be a Riemannian metric in $M$,
and $[g]$   the conformal class of $g$.
We consider the functional $\mathcal{E}$
defined on the space of Riemannian metrics $\mathcal{M}$:
$$\mathcal{E}(g):\mathcal{M}\to\mathbb{R},~~g\mapsto\mathcal{E}(g):=\frac{\int_MR_gdV_g+2\int_{\partial M}H_gdA_g}{\mbox{Vol}_g(\partial M)^{\frac{n-2}{n-1}}},$$
where
$R_g$, $H_g$, $dV_g$, $dA_g$, $\mbox{Vol}_g(\partial M)$
denote respectively the scalar curvature of $g$, the mean curvature of $g$,
the volume measure of $g$ in $M$,
the volume measure of $g$ on $\partial M$, and the volume of $\partial M$ with respect to $g$.
In particular, the mean curvature $H_{g}$ here is defined as the trace of the second fundamental form of $(\partial M, g)$ (which is not divided by $n-1$).
We define \textit{the type II Yamabe constant} $Y_{II}(M, \partial M, [g])$ by
\[
Y_{II}(M, \partial M, [g]) := \inf_{h \in [g]} \mathcal{E}(h).
\]
\begin{rema}
  As pointed out in \cite{escobar1994addendum}, $Y_{II}(M, \partial M, [g])$ can be $-\infty$.
  Indeed, it is equal to $-\infty$ if the first eigenvalue of the conformal Laplacian of $g$ with respect to Dirichlet boundary condition is negative (see \cite{escobar1994addendum}).
\end{rema}
A metric $\tilde{g}\in [g]$ is
a \textit{Yamabe metric of type II} (or a \textit{type II Yamabe metric})
if $Y_{II}(M, \partial M, [g]) > -\infty$ and $\tilde{g}$ is a minimizer of $\mathcal{E}|_{[g]}$, i.e., $\mathcal{E}(\tilde{g}) = Y_{II}(M, \partial M, [g]) > -\infty$.
If $\tilde{g}$ is a type II Yamabe metric on $M$, then $R_{\tilde{g}} \equiv 0$ in $M$ and $H_{\tilde{g}} \equiv \frac{1}{2} \mathrm{Vol}(\partial M, \tilde{g})^{-\frac{1}{n-1}} Y_{II}(M, \partial M, [\tilde{g}])$ on $\partial M$ (see Proposition \ref{prop-crit}).
We define the functional $Q_{g}: W^{1,2}(M) \rightarrow \mathbb{R}$ by
\[
Q_{g}(\phi) := \frac{\int_{M} \frac{4(n-1)}{n-2}|\nabla \phi|^{2} + R_{g} \phi^{2}\, dV_{g} + 2 \int_{\partial M} H_{g} \phi^{2}\, dA_{g}}{(\int_{\partial M} |\phi|^{\frac{2(n-1)}{n-2}} dA_{g})^{\frac{n-2}{n-1}}}.
\]
Escobar \cite{escobar1992conformal} initially considered this functional ($Q_{g}$ is just $Q$ in \cite{escobar2004conformal} multiplied by $\frac{4(n-1)}{n-2}$) to address the following problem:
\begin{prob}[Yamabe Problem with boundary]
  Let $(M^{n}, g)$ be an $n$-dimensional, compact connected Riemannian manifold with boundary $\partial M$, $n \ge 3$.
  Is there a conformally related metric $\tilde{g}$ with zero scalar curvature and boundary of constant mean curvature?
\end{prob}
Escobar \cite{escobar1992conformal}, Marques \cite{marques2005existence, marques2007conformal} and Almaraz \cite{almaraz2010existence} gave affirmative answers for a large class of Riemannian manifolds.
It is now well-known that the type II Yamabe constant can be characterized as follows:

\begin{prop}\label{prop-equivalence} There holds
  \[
  \begin{split}
    Y_{II}(M, \partial M, [g]) &= \inf \{Q_{g}(\phi) : \phi \in C^{\infty}(M),~\phi > 0~\mathrm{on}~M \} \\
    &= \inf \{ Q_{g}(\phi) : \phi \in W^{1,2}(M),~\phi \neq 0~\mathrm{on}~\partial M \}.
  \end{split}
  \]
  Moreover, if $R_{g} = 0$ in $M$, then
  \[
  Y_{II}(M, \partial M, [g]) = \inf \{Q_{g}(\phi) : \phi \in C^{\infty}(M),~\phi > 0~\mathrm{on}~M,~\Delta_{g} \phi = 0~\mathrm{in}~M \}.
  \]
\end{prop}
\begin{proof}
  The first equality follows immediately from the definition.
  The second equality can be proven as follows:
  For any $\phi \in W^{1,2}(M)$ with $\phi \neq 0$ on $\partial M$, $|\phi| \in W^{1,2}(M)$ and $Q(|\phi|) \le Q(\phi)$.
  Hence
  \[
  \begin{split}
  &\inf \{ Q_{g}(\phi) : \phi \in W^{1,2}(M),~\phi \neq 0~\mathrm{on}~\partial M \} \\
  &= \inf \{ Q_{g}(\phi) : \phi \in W^{1,2}(M),~\phi \ge 0~\mathrm{on}~M,~\phi \neq 0~\mathrm{on}~\partial M \}.
    \end{split}
  \]
  In order to prove the second equality, it is enough to prove the following,
  \[
  \begin{split}
   \inf_{W^{1,2}} Q_{g} &:=  \inf \{ Q_{g}(\phi) : \phi \in W^{1,2}(M),~\phi \ge 0~\mathrm{on}~M,~\phi \neq 0~\mathrm{on}~\partial M \} \\
    &\ge \inf \{Q_{g}(\phi) : \phi \in C^{\infty}(M),~\phi > 0~\mathrm{on}~M \} =: \inf_{C^{\infty}} Q_{g}.
  \end{split}
  \]
  Since $\displaystyle\inf_{C^{\infty}} Q_{g} \ge \inf_{W^{1,2}} Q_{g}$, we have $\displaystyle\inf_{W^{1,2}} Q_{g} = -\infty$ whenever $\displaystyle\inf_{C^{\infty}} Q_{g} = -\infty$. In particular, the second equality holds in this case.
  Hence, we may assume that $\displaystyle\inf_{C^{\infty}} Q_{g}$ is finite.
  Suppose that $\displaystyle-\infty < \inf_{W^{1,2}} Q_{g} < \inf_{C^{\infty}} Q_{g}$.
  Then we can take a function $\phi \in W^{1,2}(M)$ with $\phi \ge 0$ on $M$ and $\phi \neq 0$ on $\partial M$ such that
  $$Q_{g}(\phi) \le \inf_{W^{1,2}} Q_{g} + \frac{1}{2} (\inf_{C^{\infty}} Q_{g} - \inf_{W^{1,2}} Q_{g}).$$
  Since $C^{\infty}(M)$ is dense in $W^{1,2}(M)$, there is a smooth nonnegative function
  $\tilde{\psi} \in C^{\infty}(M)$ such that $$Q_{g}(\tilde{\psi}) \le Q_{g}(\phi) - \frac{1}{8} (\inf_{C^{\infty}} Q_{g} - \inf_{W^{1,2}} Q_{g}).$$
  Moreover, from the continuity of the functional $Q_{g}$, there is a positive constant $a > 0$ such that
  $$Q_{g}(\tilde{\psi} + a) \le Q_{g}(\tilde{\psi}) - \frac{1}{8} (\inf_{C^{\infty}} Q_{g} - \inf_{W^{1,2}} Q_{g}).$$
Summing these up, we obtain that
\[
\inf_{C^{\infty}} Q_{g} \le Q_{g}(\tilde{\psi} + a) \le \inf_{W^{1,2}} Q_{g} + \frac{1}{4} (\inf_{C^{\infty}} Q_{g} - \inf_{W^{1,2}} Q_{g}) < \inf_{C^{\infty}} Q_{g},
\]
which is a contradiction.
When $\displaystyle\inf_{W^{1,2}} Q_{g} = -\infty$, we can take a function $\phi \in \{ Q_{g}(\phi) : \phi \in W^{1,2}(M),~\phi \ge 0~\mathrm{on}~M,~\phi \neq 0~\mathrm{on}~\partial M \}$ such that $Q_{g}(\phi) \ll \inf_{C^{\infty}} Q_{g}$, since we have assumed that $\inf_{C^{\infty}} Q_{g}$ is finite.
Therefore, this leads to a contradiction in the same way as in the previous case.
This completes the proof of the second equality.

Finally, we prove the last equality when $R_{g} = 0$ in $M$.
To this end, it is enough to show that
\begin{equation}\label{eq-claim}
  \begin{split}
  &\inf\{Q_{g}(\phi) : \phi \in C^{\infty}(M),~\phi > 0~\mathrm{on}~M \} \\
    &\ge \inf \{Q_{g}(\phi) : \phi \in C^{\infty}(M),~\phi > 0~\mathrm{on}~M,~\Delta_{g} \phi = 0~\mathrm{in}~M \}.
\end{split}
\end{equation}
Let $\phi \in \{\phi \in C^{\infty}(M),~\phi > 0~\mathrm{on}~M \}$ and $\phi_{harm}$ is the harmonic extension of $\phi|_{\partial M}$, that is, it satisfies that $\Delta_{g} \phi_{harm} = 0$ in $M$ and $\phi_{harm} |_{\partial M} = \phi|_{\partial M}$.
Then it follows that $\phi_{harm} > 0$ on $M$ by the maximum principle.
Moreover, since $\Delta_{g} \phi_{harm} = 0$ in $M$ and $\phi_{harm} |_{\partial M} = \phi|_{\partial M}$, we have
\[
\int_{M} \langle \nabla \phi_{harm}, \nabla(\phi - \phi_{harm}) \rangle\, dV_{g} = 0
\]
by integration by parts.
Therefore we have
\[
\begin{split}
  \int_{M} |\nabla \phi|^{2}\, dV_{g} &= \int_{M} |\nabla(\phi - \phi_{harm}) + \nabla \phi_{harm}|^{2}\, dV_{g} \\
  &= \int_{M} |\nabla(\phi - \phi_{harm})|^{2}\, dV_{g} + \int_{M} |\nabla \phi_{harm}|^{2}\, dV_{g} \\
&\ge \int_{M} |\nabla \phi_{harm}|^{2}\, dV_{g}.
\end{split}
\]
Hence, from the definition of the energy $Q_{g}$ and the assumption that $R_{g} = 0$ in $M$, we obtain the desired inequality (\ref{eq-claim}).
\end{proof}

It follows from  {\cite[Proposition 1.4]{escobar1992conformal}} that, if $Y_{II} (M, \partial M, [g]) > -\infty$, then there always exists a scalar-flat metric with boundary of constant mean curvature in $[g]$. Hence, by the conformal invariance of $Y_{II}(M, \partial M, [g])$ and Proposition \ref{prop-equivalence}, we can always assume that $g$ is scalar-flat when we consider its type II Yamabe constant $Y_{II}(M, \partial M, [g])$, suppose that it is finite.

It also follows from {\cite[Proposition 2.1]{escobar1992conformal}} and \cite{escobar1994addendum} that,
if $Y_{II}(M, \partial M, [g])$ is finite and $Y_{II}(M^{n}, \partial M, [g]) < Y_{II}(\mathbb{B}^{n}, \partial \mathbb{B}^{n}, [\delta])$, then there is a type II Yamabe metric $g_{0}$ in the conformal class $[g]$.
Here, $\delta$ denotes the Euclidean metric restricted on the unit $n$-ball $\mathbb{B}^{n} (\subset \mathbb{R}^{n})$.

On the other hand, $Y_{II} (M, \partial M, [g])$ is finite whenever  $R_{g} \ge 0$.
Indeed, if $R_{g} \ge 0$, then we have for all $\phi \in W^{1,2}(M)$ with $\phi \neq 0$ on $\partial M$
\[
\begin{split}
  Q_{g}(\phi) &\ge \frac{2 \int_{\partial M} H_{g} \phi^{2}\, dA_{g}}{(\int_{\partial M} |\phi|^{\frac{2(n-1)}{n-2}} dA_{g})^{\frac{n-2}{n-1}}} = \frac{2 \int_{\partial M} (H_{g}^{+} - H_{g}^{-}) |\phi|^{2}\, dA_{g}}{(\int_{\partial M} |\phi|^{\frac{2(n-1)}{n-2}} dA_{g})^{\frac{n-2}{n-1}}} \\
  &\ge \frac{-2 \int_{\partial M} H^{-}_{g} |\phi|^{2}\, dA_{g}}{(\int_{\partial M} |\phi|^{\frac{2(n-1)}{n-2}} dA_{g})^{\frac{n-2}{n-1}}} \\
  &\ge \frac{-2 \|H^{-}_{g}\|_{L^{n-1}} \|\phi\|_{L^{\frac{n-1}{n-2}}(\partial M)}}{(\int_{\partial M} |\phi|^{\frac{2(n-1)}{n-2}} dA_{g})^{\frac{n-2}{n-1}}} = -2||H^{-}_{g}||_{L^{n-1}} > -\infty,
\end{split}
\]
where $H^+_{g}=\max\{H_g,0\}$ denotes the positive part of $H_g$ and
$H^-_{g}=\max\{-H_g,0\}$ negative part of $H_{g}$. Here, we have used H\"{o}lder's inequality in the last inequality.

\begin{prop}\label{prop-crit}
  Given a metric $g$ on $M^{n}~(n \ge 3)$, if $u$ is a smooth critical point of the energy $Q_{g}$, then the metric $\bar{g} := u^{\frac{4}{n-2}} g$ satisfies $R_{\bar{g}} = 0$ in $M$ and $H_{\bar{g}} \equiv \mathrm{const}$ on $\partial M$.
  In particular, if $g$ is a type II Yamabe metric on $M$, then $R_{g} \equiv 0$ in $M$ and $H_{g} \equiv \frac{1}{2} \mathrm{Vol}(\partial M, g)^{-\frac{1}{n-1}} Y_{II}(M, \partial M, [g])$ on $\partial M$.
\end{prop}
\begin{rema}
  Note that the boundary $\partial M$ is not necessarily connected, even we have assumed that $M$ is connected.
  In particular, Proposition \ref{prop-crit}   asserts that each connected component of $\partial M$ has constant mean curvature of the same value.
\end{rema}
\begin{proof}
  Let $u \in \{ \phi \in C^{\infty}(M):\phi > 0~\mathrm{on}~M \}$ be a critical point of the energy $Q_{g}$.
From the first variation of $Q_{g}$ at $u$, we find
  \[
  \begin{split}
 0= &\frac{2 \int_{M} \frac{4(n-1)}{n-2} \langle \nabla \phi, \nabla u \rangle + R_{g} \phi \cdot u\, dV_{g} + 4 \int_{\partial M} H_{g} \phi \cdot u\, dA_{g}}{(\int_{\partial M} |\phi|^{\frac{2(n-1)}{n-2}} dA_{g})^{\frac{n-2}{n-1}}} \\
  &-2\, \frac{\int_{M} \frac{4(n-1)}{n-2}|\nabla u|^{2} + R_{g} u^{2}\, dV_{g} + 2 \int_{\partial M} H_{g} u^{2}\, dA_{g}}{(\int_{\partial M} u^{\frac{2(n-1)}{n-2}} dA_{g})^{\frac{n-2}{n-1}}} \cdot \frac{\int_{\partial M} \phi \cdot u^{\frac{n}{n-2}}\, dA_{g}}{\int_{\partial M} u^{\frac{2(n-1)}{n-2}} dA_{g}}
    \end{split}
  \]
  for all $\phi \in C^{\infty}(M)$.
  From this, we get
  \[
  \begin{split}
  -\frac{4(n-1)}{n-2}\Delta_{g}u+R_{g}u&= 0~~\mbox{ in }M,\\
\frac{2(n-1)}{n-2}\frac{\partial u}{\partial\nu_{g}}+H_{g}u&=c u^{\frac{n}{n-2}}~~\mbox{ on }\partial M,
    \end{split}
  \]
  where $c := \frac{1}{2} ( \int_{M} \frac{4(n-1)}{n-2}|\nabla u|^{2} + R_{g} u^{2}\, dV_{g} + 2 \int_{\partial M} H_{g} u^{2}\, dA_{g}) (\int_{\partial M} u^{\frac{2(n-1)}{n-2}} dA_{g})^{-1}$.
  Here, $\frac{\partial}{\partial \nu_{g}}$ denotes the outward normal derivative with respect to $g$.
  Then, from the formula of the scalar curvature and the mean curvature under the conformal change of metrics,  we obtain that $R_{\bar{g}} = 0$ in $M$ and $H_{\bar{g}} \equiv c$ on $\partial M$, where $\bar{g} = u^{\frac{4}{n-2}} g$.
  In particular, when $g$ itself is a type II Yamabe metric, $\phi \equiv 1$ is a critical point of $Q_{g}$,
  and the last assertion follows from this and the definition of $Y_{II}(M, \partial M, [g])$.
\end{proof}

\begin{prop}\label{prop-nonpositive}
Let $(M^{n}, g)$ be a compact connected Riemannian manifold with boundary of dimension $n \ge 3$.\\
(1) Suppose that $h, \bar{h}\in [g]$  are both scalar-flat
and their mean curvatures on the boundary $\partial M$ are
nonpositive constant of same sign (i.e.  $H_h$, $H_{\bar{h}}$ are both negative
constants, or $H_h=H_{\bar{h}}=0$).
Then $h$ is equal to $\bar{h}$ up to rescaling.\\
(2) If $-\infty < Y_{II}(M^{n}, \partial M^{n}, [g]) \le 0$, then there is a unique (up to rescaling) type II Yamabe metric $g_{0}$ in $[g]$.\\
Here, $g$ is equal to $\bar{g}$ up to rescaling means that $g = c \cdot \bar{g}$ for some positive constant $c$.
\end{prop}
\begin{proof}
\underline{(1)}:
We first consider the case when $H_h$, $H_{\bar{h}}$ are both negative.
  Since $H_{h}$ and $H_{\bar{h}}$ are negative constants, we can take a positive constant $c > 0$ such that
  $H_{c \bar{h}} = c^{-1/2} H_{\bar{h}} = H_{h}$ on $\partial M$.
  Replacing $\bar{h}$ by $c\bar{h}$, we may assume that
\begin{equation}\label{4}
H_{\bar{h}}= H_{h}~~\mbox{ on }\partial M.
\end{equation}
  Since $h$ and $\bar{h}$  are in the same conformal class $[g]$, there is a smooth positive function $u$ on $M$ such that $\bar{h} = u^{\frac{4}{n-2}} h$.
  The goal is showing that $u \equiv \mathrm{const} > 0$ on $M$.
  From the conformal changing formula of the scalar curvatures and the assumptions that $R_{h}=R_{c\bar{h}}=0$, we have that $\Delta_{h} u = 0$ in $M$.
  Take a point $x_{0} \in M$ where the function $u$ attains the maximum, i.e., $u(x_{0}) = \max_{M} u.$
  Then $u(x_{0}) = \max_{M} u = \max_{\partial M} u$ and $\frac{\partial u}{\partial \nu_{h}} (x_{0}) \ge 0$ by the Hopf maximum principle.
  Here, $\frac{\partial}{\partial \nu_{h}}$ denotes the outward normal derivative with respect to $h$.
    From (\ref{4}) and the conformal changing formula of mean curvatures, we have
  \[
  \frac{2(n-1)}{n-2} \frac{\partial u}{\partial \nu_{h}}(x_{0}) + H_{h} \cdot u(x_{0}) = H_{\bar{h}} \cdot u(x_{0})^{\frac{n}{n-2}}~\mathrm{on}~\partial M.
  \]
  Hence we have $H_{\bar{h}} (u(x_{0})^{\frac{n}{n-2}} - c^{-1/2} u(x_{0})) \ge 0$ on $\partial M$. Since $H_{\bar{h}} < 0$, we have $u(x_{0})^{\frac{2}{n-2}} \le c^{-1/2}$.
  Similarly, if a point $y_{0} \in M$ where $u$ attains the minimum, then $u(y_{0}) = \min_{\partial M} u$ and $c^{-1/2} \le u(y_{0})^{\frac{2}{n-2}}$.
  Therefore we obtain that $u \equiv c^{-1/2}$ on $M$.

We next consider the case when $H_h=H_{\bar{h}}=0$.
By the same argument as above, we have
  \[
  \Delta_{h} u = 0~\mathrm{in}~M~\mathrm{and}~\frac{\partial u}{\partial \nu_{h}} = 0~\mathrm{on}~\partial M.
  \]
  By integration by parts, we have $\int_{M} |\nabla_{h} u|^{2}\, dV_{h} = 0.$
  Hence, $u \equiv \mathrm{const} > 0$ on $M$, since $M$ is connected.

\noindent
  \underline{(2)}:
  Since $-\infty < Y_{II}(M^{n}, \partial M^{n}, [g]) \le 0 < Y_{II}(\mathbb{B}^{n}, \partial \mathbb{B}^{n}, [\delta])$, from \cite{escobar1992conformal} (see also \cite{escobar1994addendum}), there is a type II Yamabe metric $g_{0} \in [g].$
  Since a type II Yamabe metric is a scalar-flat metric with boundary of constant non-positive mean curvature (by Proposition \ref{prop-crit}), the uniqueness (up to rescaling) follows from $(1)$.
\end{proof}
On the other hand, for the case of $Y_{II}(M, \partial M, [g]) > 0$, little is known about the uniqueness.
In \cite{escobar1990uniqueness}, Escobar classified all positive solutions of the following equation:
\[
\begin{cases}
  \Delta_{\delta} u = 0~~\mathrm{in}~\mathbb{B}^{n}, \\
  \frac{\partial u}{\partial \nu_{\delta}} + \frac{n-2}{2} u = \left( \frac{n-2}{2} \right)^{2} u^{\frac{n}{n-2}}~~\mathrm{on}~\mathbb{S}^{n-1}.
\end{cases}
\]
Here, $\delta$ denotes the Euclidean metric on the closed unit ball $\mathbb{B}^{n}$ and $\frac{\partial}{\partial \nu_{\delta}}$ denotes the outward normal derivative with respect to $\delta$.
Indeed, all positive solutions are written as the form
\[
u_{a}(x) = \left[ \frac{2}{n-2} \cdot \frac{1-|a|^{2}}{1 + |a|^{2}|x|^{2} - 2 x \cdot a} \right]^{\frac{n-2}{2}}
\]
for some $a \in \mathbb{B}^{n}$.
C{\'a}rdenas--Sierra \cite{cardenas2019uniqueness} showed that every positive conformal class sufficiently close to a non-degenerate metric contains unique scalar-flat metric with boundary of constant mean curvature up to rescaling.
Here, a metric $g$ on a compact $n$-manifold $M^{n}$ with boundary of constant mean curvature $H_{g}$ is called \textit{non-degenerate} if either $H_{g} = 0$ or $\frac{H_{g}}{n-1}$ is not a Steklov eigenvalue.
Note here that the Euclidean metric $\delta$ on $\mathbb{B}^{n}$ is not non-degenerate by {\cite[Section 1.3]{girouard2017spectral}}.

The following is our main theorem, which provides a way to determine whether a given metric is a type II Yamabe metric or not.
\begin{theorem}\label{thm1}
Let $g$ be a type II Yamabe metric
on a compact, connected, $n$-dimensional smooth manifold
with boundary $\partial M$, where $n\geq 3$,
such that its mean curvature $H_g>0$ on $\partial M$.
Assume that $h$
is a scalar-flat metric in $M$ with positive constant mean curvature
and that $dA_{g} \le C^{\frac{n-1}{2}} dA_{h}$ for a positive constant $C > 0$ such that
\begin{equation}\label{eq-assumption-thm1}
  C \le \min \left\{ \gamma^{-\frac{2}{n-1}},~\left( \frac{H_{h}}{H_{g}} \right)^{\frac{2n-3}{n-2}} \gamma^{\frac{2}{n-2}} \right\},~~\mbox{ or }~~dA_{g} = \left( \frac{H_{h}}{H_{g}} \right)^{\frac{n-1}{2}} dA_{h},
\end{equation}
where $\gamma = \mathrm{Vol}(M, h) \cdot \mathrm{Vol}(M, g)^{-1}$.
If
\begin{equation}\label{1}
H_h  h\leq H_g g~~\mathrm{on}~M,
\end{equation}
then $h$ is also a type II Yamabe metric.
Moreover, if
\begin{equation}\label{2}
H_h  h<H_g g~~\mathrm{on}~M,
\end{equation}
then $h$ is a unique type II Yamabe metric (up to positive constant)
in $[h]_0$, where
$$[h]_0:=\{ u^{\frac{4}{n-1}} h~|~u \in C^{\infty}(M),~u > 0~\mathrm{on}~M,~\Delta_{h} u = 0~\mathrm{in}~M \}.$$
\end{theorem}
\begin{rema}
  When $H_{h} \le 0$, $h$ is a unique type II Yamabe metric by Proposition \ref{prop-nonpositive}. So it is enough to consider the case that $H_{h}$ is positive constant when we investigate the uniqueness of type II Yamabe metrics.
\end{rema}

\begin{proof}[Proof of Theorem \ref{thm1}]
From {\cite[TH\'{E}OR\`{E}ME]{banyaga1974volume}}, there is a diffeomorphism $\varphi : M \rightarrow M$ such that $dV_{\varphi^{*} h} = \gamma dV_{g}$ and $\varphi = \mathrm{id}$ on $\partial M$.
Since the scalar curvature and the mean curvature are preserved under the pullback action of diffeomorphisms,
it suffices to consider the case that $\varphi=id_M$.

From Proposition \ref{prop-equivalence}, it is enough to consider in the case
when  $\overline{h}=u^{\frac{4}{n-2}}h\in [h]_{0}$.
Since $\overline{h}=u^{\frac{4}{n-2}}h$, we have
\begin{equation}\label{3}
\begin{split}
-\frac{4(n-1)}{n-2}\Delta_{h}u+R_{h}u&=R_{\overline{h}}u^{\frac{n+2}{n-2}}~~\mbox{ in }M,\\
\frac{2(n-1)}{n-2}\frac{\partial u}{\partial\nu_{h}}+H_{h}u&=H_{\overline{h}}u^{\frac{n}{n-2}}~~\mbox{ on }\partial M,
\end{split}
\end{equation}
where $\Delta_{h}$ denotes the Laplacian of $h$,
and $\frac{\partial}{\partial\nu_{h}}$ denotes the outward normal derivative with respect to $h$.
Since $h$ is a scalar-flat metric and $\Delta_{h} u = 0$ in $M$, it follows that $\overline{h}$ is scalar-flat as well.
Hence, we compute
\begin{equation*}
\begin{split}
\mathcal{E}(\overline{h})
&=\frac{\int_MR_{\overline{h}}dV_{\overline{h}}+2\int_{\partial M}H_{\overline{h}}dA_{\overline{h}}}{\mbox{Vol}_{\overline{h}}(\partial M)^{\frac{n-2}{n-1}}}=\frac{2\int_{\partial M}H_{\overline{h}}dA_{\overline{h}}}{\mbox{Vol}_{\overline{h}}(\partial M)^{\frac{n-2}{n-1}}}\\
&=\frac{2\int_{\partial M}u^{-\frac{n}{n-2}}(\frac{2(n-1)}{n-2}\frac{\partial u}{\partial\nu_{h}}+H_hu)u^{\frac{2(n-1)}{n-2}}dA_h}{(\int_{\partial M}u^{\frac{2(n-1)}{n-2}}dA_h)^{\frac{n-2}{n-1}}}\\
&=\frac{2\int_{\partial M}(\frac{2(n-1)}{n-2}u\frac{\partial u}{\partial\nu_{h}}+H_hu^2)dA_h}{(\int_{\partial M}u^{\frac{2(n-1)}{n-2}}dA_h)^{\frac{n-2}{n-1}}},
\end{split}
\end{equation*}
where we have used the second equation of (\ref{3})
and the fact that $\overline{h}$ is scalar-flat.
By integration by parts and the fact that $\Delta_{h} u = 0$ in $M$, we obtain
\begin{equation*}
\mathcal{E}(\overline{h})
=\frac{\int_M\frac{4(n-1)}{n-2}|\nabla_hu|_{h}^2dV_h+2\int_{\partial M}H_hu^2dA_h}{(\int_{\partial M}u^{\frac{2(n-1)}{n-2}}dA_h)^{\frac{n-2}{n-1}}}.
\end{equation*}
It follows from the assumption (\ref{1}) that
$$\frac{H_h}{H_g}|\nabla_g u|_{g}^2\leq |\nabla_h u|_{h}^2.$$
This together with the assumptions (\ref{1}), (\ref{eq-assumption-thm1}) and
$dV_{h}=\gamma dV_g$ gives
\begin{equation*}
\begin{split}
\mathcal{E}(\overline{h})
&=\frac{\int_M\frac{4(n-1)}{n-2}|\nabla_hu|_{h}^2\gamma dV_g+2\int_{\partial M}H_hu^2 dA_h}{(\int_{\partial M}u^{\frac{2(n-1)}{n-2}} dA_h)^{\frac{n-2}{n-1}}}\\
&\ge \frac{\int_M\frac{4(n-1)}{n-2}|\nabla_hu|_{h}^2\gamma dV_g+2\int_{\partial M}H_hu^2 C^{-\frac{n-1}{2}}~dA_{g}}{\left( \int_{\partial M}u^{\frac{2(n-1)}{n-2}} \left( \frac{H_{g}}{H_{h}} \right)^{\frac{n-1}{2}}~dA_g \right)^{\frac{n-2}{n-1}}} \\
&\geq \min \{ \gamma, C^{-\frac{n-1}{2}}\} \left( \frac{H_h}{H_g} \right)^{\frac{n}{2}}
\frac{\int_M\frac{4(n-1)}{n-2}|\nabla_gu|_{g}^2dV_g+2\int_{\partial M}H_gu^2dA_g}{(\int_{\partial M}u^{\frac{2(n-1)}{n-2}}dA_g)^{\frac{n-2}{n-1}}}\\
&=\min \{ \gamma, C^{-\frac{n-1}{2}}\} \left( \frac{H_h}{H_g} \right)^{\frac{n}{2}}\mathcal{E}(u^{\frac{4}{n-2}}g).
\end{split}
\end{equation*}
Since $g$ is a type II Yamabe metric, we have
\begin{equation}\label{5}
\mathcal{E}(\overline{h})\geq\min \{ \gamma, C^{-\frac{n-1}{2}}\} \left( \frac{H_h}{H_g} \right)^{\frac{n}{2}}\mathcal{E}(g).
\end{equation}
Note that the right hand side of (\ref{5}) is equal to
\begin{equation*}
\begin{split}
&\min \{ \gamma, C^{-\frac{n-1}{2}}\} \left( \frac{H_h}{H_g} \right)^{\frac{n}{2}}\mathcal{E}(g)\\
&=\min \{ \gamma, C^{-\frac{n-1}{2}}\} \left( \frac{H_h}{H_g} \right)^{\frac{n}{2}}\frac{2\int_{\partial M}H_gdA_g}{(\int_{\partial M}dA_g)^{\frac{n-2}{n-1}}} \\
&\ge\min \{ \gamma, C^{-\frac{n-1}{2}}\} \left( \frac{H_h}{H_g} \right)^{\frac{n-2}{2}} \left( \frac{H_{h}}{H_{g}} \right)^{\frac{n-1}{2}} C^{-\frac{n-2}{2}}\frac{2\int_{\partial M}H_h dA_h}{(\int_{\partial M} dA_h)^{\frac{n-2}{n-1}}}.
\end{split}
\end{equation*}
Hence, if
\begin{equation}\label{eq-preassumption}
  C \le \min \left\{ \gamma^{-\frac{2}{n-1}},~\left( \frac{H_{h}}{H_{g}} \right)^{\frac{2n-3}{n-2}} \gamma^{\frac{2}{n-2}} \right\}~~\mbox{ or }~~\gamma^{-\frac{2}{n-1}} \le C \le \frac{H_{h}}{H_{g}},
\end{equation}
then it follows that
\[
\min \{ \gamma, C^{-\frac{n-1}{2}}\} \left( \frac{H_h}{H_g} \right)^{\frac{n-2}{2}} \left( \frac{H_{h}}{H_{g}} \right)^{\frac{n-1}{2}} C^{-\frac{n-2}{2}} \ge 1.
\]
Therefore we finally obtain that $\mathcal{E}(\overline{h}) \ge \mathcal{E}(h)$, and it follows from the definition of the type II Yamabe metrics that
$h$ is a type II Yamabe metric in $[h]$.
Note that the second condition in (\ref{eq-preassumption}) is equivalent to the condition that $dA_{g} = \left( H_{h} \cdot H_{g}^{-1} \right)^{\frac{n-1}{2}} dA_{h}$, since we have assumed (\ref{1}).

Furthermore, if we assume that $\overline{h}$ is also a type II Yamabe metric,
then we must have $\mathcal{E}(\overline{h})=\mathcal{E}(h)$.
In particular, all the above inequalities become equalities.
Hence, if we further assume (\ref{2}), we must have
$\nabla_g u=0$, which implies that $u\equiv \mathrm{const}$ on $M$,
since $M$ is connected.
In other words, $h$ is a unique Yamabe metric up to positive constant in $[h]_0$.
\end{proof}

\begin{cor}
  Let $M, g, h$ be as in Theorem \ref{thm1}.
  Assume that $\mathrm{Vol}(M, h) = \mathrm{Vol}(M, g)$ (i.e., $\gamma =1$) and \eqref{1} in Theorem \ref{thm1}.
  If $H_{g} \ge H_{h}$ and
  \[
  \sqrt{\det(g|_{\partial M})} \cdot H_{g}^{\frac{(n-1)(2n-3)}{2(n-2)}} \le \sqrt{\det(h|_{\partial M})} \cdot H_{h}^{\frac{(n-1)(2n-3)}{2(n-2)}}~\mathrm{on}~\partial M,
  \]
  then $h$ is also a type II Yamabe metric.
  Here, $g|_{\partial M}$ and $h|_{\partial M}$ respectively denote the induced metrics on $\partial M$ of $g$ and $h$.
  Moreover, if \eqref{2} is assumed instead of \eqref{1}, then $h$ is a unique type II Yamabe metric in its conformal class.
\end{cor}
\begin{proof}
One can easily check that the assumptions of this corollary imply the first condition of (\ref{eq-assumption-thm1}) and (\ref{1}) in Theorem \ref{thm1}.
  Therefore the assertion follows from Theorem \ref{thm1}.
\end{proof}

\begin{cor}
  Let $M$ be as in Theorem \ref{thm1} and $\{ g_{t}~|~T \le t \le T' \}$ a smooth variation of metrics satisfying the following conditions:
  \begin{itemize}
    \item[(1)] $R_{g_{t}} = 0$ in $M$ and $H_{g_{t}} = \mathrm{const}$ on $\partial M$ for all $t \in [T, T']$,
    \item[(2)] $g_{T}$ is a type II Yamabe metric,
    \item[(3)] $H_{g_{t}} > 0$ for all $t \in [T, T')$,
    \item[(4)] $H_{g_{T'}} = 0$ on $\partial M$,
    \item[(5)] $C^{\frac{n-1}{2}} dA_{g_{T'}} > dA_{g_{T}}$ for a positive constant $C > 0$ such that $C \le \gamma^{-\frac{2}{n-1}},$
where $\gamma = \mathrm{Vol}(M, g_{T'}) \cdot \mathrm{Vol}(M, g_{T})^{-1}$.
  \end{itemize}
  Then there exists a positive constant $\delta > 0$ such that $g_{t}$ is also a type II Yamabe metric for every $t \in [T' - \delta, T')$.
\end{cor}
\begin{proof}
  By the assumptions of the corollary, $g_{t}$ satisfies the conditions (\ref{eq-assumption-thm1}) and (\ref{1}) in Theorem \ref{thm1} when $t$ is sufficiently close to $T'$.
  Hence we can apply Theorem \ref{thm1} to such $g_{t}$ and obtain the desired assertion.
\end{proof}

\section{Obata-type theorem}\label{section-obata}

In \cite{obata1962certain, obata1971conjectures}, Obata proved the following uniqueness theorem for constant scalar curvature metrics on a closed Einstein manifold.
\begin{theorem}[\cite{obata1962certain, obata1971conjectures}]
Let $g$ be an Einstein metric on a closed $n$-manifold ($n \ge 2$), and $\bar{g} \in [g]$ a constant scalar curvature metric. Then the following holds:
  \begin{itemize}
    \item[(1)] If $(M, [g])$ is conformally equivalent to $(\mathbb{S}^{n}, [g_{std}])$, then there exists a homothety $\Phi : (\mathbb{S}^{n}, [g_{std}]) \rightarrow (M, g)$ and a conformal transformation $\varphi$ such that
    $\Phi^{*} \bar{g} = \varphi^{*} (\Phi^{*} g)$.
    \item[(2)] If $(M, [g])$ is not conformally equivalent to $(\mathbb{S}^{n}, [g_{std}])$, then $\bar{g}$ is equal to $g$ up to rescaling.
  \end{itemize}
  Here, $g_{std}$ denotes the round metric on the $n$-sphere $\mathbb{S}^{n}$ of constant sectional curvature one.
\end{theorem}
On compact manifolds with boundary, some Obata-type theorems were proved for Escobar's first problem \cite{akutagawa2021obata, escobar1990uniqueness}.
In view of these, a natural question is the following:
\begin{prob}\label{prob-Obata}
Does an Obata-type theorem hold in our case? More precisely, for a scalar-flat metric $g_{0}$ on a compact manifold with totally umbilic boundary,
 if $g\in[g_0]$ such that $R_{g} = 0$ in $M$ and $H_{g}$ is constant on $\partial M$, then is it true that $g = c g_{0}$ for some constant $c>0$?
\end{prob}

It follows from Proposition \ref{prop-nonpositive} $(1)$ that this is true in the nonpositive case, i.e.
if such metrics $g$ and $g_{0}$ have the nonpositive mean curvatures of same sign, then $g$ is equal to $g_{0}$ up to rescaling.

\subsection{Examples of non-type II Yamabe scalar-flat metrics with umbilic boundary of dimension $4, 5,$ or $6$}
We are going to apply {\cite[Th\'{e}or\`{e}me 5]{cherrie1984problemes}} to construct some examples of metrics which are conformal equivalent to the standard flat metric on the unit $n$-ball $\mathbb{B}^{n}$ where $n = 4, 5$ or $6,$ such that their scalar curvature is flat in $\mathbb{B}^{n}$ and mean curvature is  not constant on $\partial \mathbb{B}^{n}.$
We have the following:

\begin{theorem}[A special case of {\cite[Th\'{e}or\`{e}me 5]{cherrie1984problemes}}]\label{theorem-cherrier}
 Let $M^{n}$ be a compact connected smooth $n$-manifold with boundary $\partial M$ of dimension $n \ge 3$ and $g_{0}$ a Riemannian metric on $M.$
For a smooth function $v$ and a positive smooth function $v' > 0$ on $\partial M,$ if
\begin{equation}\label{eq-sufficient}
  \frac{\sigma}{\tau} \cdot \sup_{\partial M} v' \cdot \left( C(n-1, \tau, g_{0})^{2} \mu_{\sigma, \tau} \right)^{\tau/2} < 1,
\end{equation}
then there is a positive smooth function $\varphi$ on $\overline{M}$ and a constant $\lambda \in \mathbb{R}$ such that the following  holds:
\begin{equation}\label{eq-conformal}
  \begin{cases}
  \Delta_{g_{0}} \varphi = 0~\mathrm{in}~M, \\
  \frac{\partial \varphi}{\partial \nu_{g_{0}}} + \frac{n-2}{2(n-1)} v \varphi = \lambda \frac{n-2}{2(n-1)} v' \varphi^{\frac{n}{n-2}}~\mathrm{on}~\partial M.
\end{cases}
\end{equation}
Here, $\sigma = \frac{2n}{n-2},~\tau = \frac{2(n-1)}{n-2},$
\[
\begin{split}
  &C(n-1, \tau, g_{0}) \\
  &:= \sup \Bigg\{ C > 0:\mathrm{There~is~a~constant}~\tilde{C} = \tilde{C}(C, M, g_{0}) > 0~\mathrm{such~that} \\
  &\hspace{8mm}\left( \int_{\partial M} |\psi|^{\tau} dA_{g_{0}} \right)^{\frac{2}{\tau}} \le C \int_{M} |\nabla \psi|^{2} dV_{g_{0}} + \tilde{C} \int_{\partial M} \psi^{2} dA_{g_{0}}~\mathrm{for~all}~\psi \in H^{1}(M) \Bigg\}
\end{split}
\]
and
\[
\mu_{\sigma, \tau} := \inf_{\phi \in H^{1}(M),~\Gamma_{\sigma, \tau} (\phi) = 1} \int_{M} |\nabla \phi|^{2} dV_{g_{0}} + \frac{n-2}{2(n-1)} \int_{\partial M} v \phi^{2} dA_{g_{0}},
\]
where $\Gamma_{\sigma, \tau} (\phi) := \frac{n-2}{2(n-1)} \cdot \frac{\sigma}{\tau} \int_{\partial M} v' |\phi|^{\tau} dA_{g_{0}}$.
\end{theorem}
\begin{rema}
  The original definition of $C(n-1, \tau, g_{0})$ in \cite{cherrie1984problemes} is slightly different from the one given here (i.e., $\int_{\partial M} \psi^{2} dA_{g_{0}}$ is replaced with $\int_{M} \psi^{2} dV_{g_{0}}$). But one can see from the proof of {\cite[Th\'{e}or\`{e}me 5]{cherrie1984problemes}}
  that Theorem \ref{theorem-cherrier} still holds.
\end{rema}

We will apply this theorem to the case that $(M, \partial M, g_{0})$ is the standard Euclidean unit $n$-ball $(\mathbb{B}^{n}, \partial \mathbb{B}^{n}, g_{0})$, 
where $n\geq 3$.
In this case, $R_{g_{0}} = 0$ in $\mathbb{B}^{n}$ and $H_{g_{0}} = n-1$ on $\partial \mathbb{B}^{n}.$
We take $v \equiv n-1$ on $\partial \mathbb{B}^{n}.$

First of all, we consider the case that $v'$ is a positive constant $c > 0$.
Examining the condition under which a constant function $\phi_{0}$ satisfies
\[
\Gamma_{\sigma, \tau} (\phi_{0}) = \frac{n-2}{2(n-1)} \cdot \frac{\sigma}{\tau} \int_{\partial \mathbb{B}^{n}} c~|\phi_{0}|^{\tau} dA_{g_{0}} = 1,
\]
we find that  
\[
\phi_{0} = \left( \frac{n-2}{2(n-1)} \cdot \frac{\sigma}{\tau} c \mathrm{Vol}(\partial \mathbb{B}^{n}, g_{0}) \right)^{-\frac{1}{\tau}}.
\]
Therefore, for such a constant function $\phi_{0}$, we have $\Gamma_{\sigma, \tau}(\phi_{0}) = 1$ and
\[
\int_{\partial \mathbb{B}^{n}} \phi_{0}^{2}~dA_{g_{0}} = \mathrm{Vol}(\partial \mathbb{B}^{n})^{\frac{1}{n-1}} \left(\frac{n-2}{2(n-1)} \cdot \frac{\sigma}{\tau} c \right)^{-\frac{n-2}{n-1}}.
\]
From \cite{beckner1993sharp}, the relative Sobolev constant $C(n-1, \tau, g_{0})$ of $(\mathbb{B}^{n}, \partial \mathbb{B}^{n}, g_{0})$ defined above satisfies
\[
C(n-1, \tau, g_{0}) \le \frac{2}{n-2} \mathrm{Vol}(\partial \mathbb{B}^{n}, g_{0})^{-\frac{1}{n-1}}.
\]
Indeed, it follows from {\cite[(36) in Theorem 4]{beckner1993sharp}}  that for all $\psi \in H^{1}(M)$
\[
\left( \int_{\partial \mathbb{B}^{n}} |\psi|^{\tau} dA_{g_{0}} \right)^{\frac{2}{\tau}} \le \mathrm{Vol}(\partial \mathbb{B}^{n}, g_{0})^{-\frac{1}{n-1}} \left( \frac{2}{n-2} \int_{\mathbb{B}^{n}} |\nabla \psi_{\mathrm{harm}}|^{2} dV_{g_{0}} + \int_{\partial \mathbb{B}^{n}} \psi^{2} dA_{g_{0}} \right),
\]
where $\psi_{\mathrm{harm}}$ is the harmonic extension of $\psi|_{\partial \mathbb{B}^{n}}$.
Since $\psi - \psi_{\mathrm{harm}} = 0$ on $\partial \mathbb{B}^{n}$ and $\Delta \psi_{\mathrm{harm}} = 0$ in $\mathbb{B}^{n},$  we have
\[
\int_{\mathbb{B}^{n}} \langle \nabla \psi_{\mathrm{harm}}, \nabla (\psi - \psi_{\mathrm{harm}}) \rangle_{g_{0}}~dV_{g_{0}} = 0
\]
by integration by parts.
This implies $\int_{\mathbb{B}^{n}} |\nabla \psi_{\mathrm{harm}}|^{2} dV_{g_{0}} \le \int_{\mathbb{B}^{n}} |\nabla \psi|^{2} dV_{g_{0}}.$
Hence, we get
\[
\left( \int_{\partial \mathbb{B}^{n}} |\psi|^{\tau} dA_{g_{0}} \right)^{\frac{2}{\tau}} \le \mathrm{Vol}(\partial \mathbb{B}^{n}, g_{0})^{-\frac{1}{n-1}} \left( \frac{2}{n-2} \int_{\mathbb{B}^{n}} |\nabla \psi|^{2} dV_{g_{0}} + \int_{\partial \mathbb{B}^{n}} \psi^{2} dA_{g_{0}} \right).
\]
Combining all these yields
\[
\begin{split}
  \frac{\sigma}{\tau} \cdot \sup_{\partial M} v' \cdot \left( C(n-1, \tau, g_{0})^{2} \mu_{\sigma, \tau} \right)^{\tau/2}
  &= \frac{\sigma}{\tau} \cdot c \cdot \left( C(n-1, \tau, g_{0})^{2} \mu_{\sigma, \tau} \right)^{\tau/2} \\
  &\le \frac{\sigma}{\tau} c \cdot \left( \frac{n-2}{2} \int_{\partial \mathbb{B}^{n}} \phi_{0}^{2} dA_{g_{0}} \right)^{\frac{\tau}{2}} \cdot C(n-1, \tau, g_{0})^{\tau} \\
  &= \frac{2(n-1)}{n-2} \cdot \left( \frac{2}{n-2} \right)^{\frac{n-1}{n-2}} \cdot \mathrm{Vol}(\partial \mathbb{B}^{n}, g_{0})^{-\frac{1}{n-2}} \\
  &= \frac{2(n-1)}{n-2} \cdot \left( \frac{2}{n-2} \right)^{\frac{n-1}{n-2}} \cdot \frac{\Gamma \left( \frac{n}{2} \right)^{\frac{1}{n-2}}}{2^{\frac{1}{n-2}} \pi^{\frac{n}{2(n-2)}}}.
\end{split}
\]
Here, $\Gamma(m)$ denotes the Gamma function.
In particular, if $n = 4, 5,$ or $6$, we have respectively
\[
\begin{split}
  \frac{\sigma}{\tau} \cdot \sup_{\partial M} v' \cdot \left( C(3, \tau, g_{0})^{2} \mu_{\sigma, \tau} \right)^{\tau/2} &\le \frac{3}{\pi} < 1, \\
  \frac{\sigma}{\tau} \cdot \sup_{\partial M} v' \cdot \left( C(4, \tau, g_{0})^{2} \mu_{\sigma, \tau} \right)^{\tau/2} &\le \left( \frac{2^{10}}{3^{6} \pi^{2}} \right)^{1/3} < 1, \\
  \frac{\sigma}{\tau} \cdot \sup_{\partial M} v' \cdot \left( C(5, \tau, g_{0})^{2} \mu_{\sigma, \tau} \right)^{\tau/2} &\le \left( \frac{3 \cdot 5^{4}}{2^{9} \pi^{3}} \right)^{1/4} < 1.
\end{split}
\]

Next, we consider the case that $n = 4, 5,$ or $6$ and $v' = c + f$, where $c > 0$ is a positive constant and $f$ is a smooth non-constant function on $\partial \mathbb{B}^{n}$
such that $0 < f \le 1$ and the volume  of $\mathrm{supp} f$ with respect to $(\partial \mathbb{B}^{n}, g_{0})$ is sufficiently small.
Then, with the same arguments as above, we can take $\mathrm{supp} f$ sufficiently small so that  
\[
\frac{\sigma}{\tau} \cdot \sup_{\partial M} v' \cdot \left( C(n-1, \tau, g_{0})^{2} \mu_{\sigma, \tau} \right)^{\tau/2} < 1.
\]

As a result, from Theorem \ref{theorem-cherrier} above, we have obtained a metric $g = \varphi^{\frac{4}{n-2}} g_{0}$ on the unit $n$-ball $\mathbb{B}^{n}$ for $n = 4, 5$ or $6$ such that $R_{g} = 0$ in $\mathbb{B}^{n}$, $\partial \mathbb{B}^{n}$ is umbilic with respect to $g$, and $H_{g}$ is not constant on $\partial \mathbb{B}^{n}$ (which corresponds to the case of $\lambda \neq 0$) or $H_{g} \equiv 0$ on $\partial M$ (which corresponds to the case of $\lambda = 0$).
However, the latter case cannot occur. Indeed, from \cite{beckner1993sharp}, \cite{escobar1988sharp} or \cite{escobar1990uniqueness}, the standard metric $g_{0}$ on $\mathbb{B}^{n}$ is a type II Yamabe metric and $\mathcal{E}(g_{0}) > 0.$
Therefore, if the latter case occur, then we have $0 = \mathcal{E}(\varphi^{\frac{4}{n-2}} g_{0}) \ge \mathcal{E}(g_{0}) > 0$, which is absurd.
Moreover, we can see from the second equation in (\ref{eq-conformal}) that the conformal factor $\varphi$ is not a constant function.

Of course, by the above arguments, we can obtain a metric $g = \varphi^{\frac{4}{n-2}} g_{0}$ on the unit $n$-ball $\mathbb{B}^{n}$ for $n = 4, 5$ or $6$ such that $R_{g} = 0$ in $\mathbb{B}^{n}$, $H_{g}$ is a positive constant $c > 0$ on $\partial \mathbb{B}^{n}$ and $\partial \mathbb{B}^{n}$ is umbilic with respect to $g$.
However, we cannot determine whether the solution $\varphi$ is constant.

On the other hand, Escobar and Garcia {\cite[Proposition 2.8]{escobar2004conformal}} showed that any smooth nonnegative function on $\partial \mathbb{B}^{n}~(n \ge 3)$ can be $C^{0, \beta}$-approximated $(0 < \beta < 1)$ by a sequence of smooth positive functions which are the mean curvature of smooth metrics of the form $u^{\frac{4}{n-2}} g_{0}$ with zero scalar curvature.

\subsection{Examples of non-type II Yamabe scalar-flat metrics with umbilic boundary (with two boundary components)}\label{subsection-sch}
Fix a constant $r \in (0, 1)$ and consider the metric $g_{r, m} := ( 1 + \frac{m}{2 |x|^{n-2}} )^{\frac{4}{n-2}} \delta$ defined on $A_{r, 1} := \{ x \in \mathbb{R}^{n}~|~r \le |x| \le 1 \}$,
where $m, |x|$ and $ \delta$ respectively denote a nonnegative constant, the Euclidean distance function from the origin and the Euclidean metric on $\mathbb{R}^{n}$.
Note that the metric $g_{r, m}$ is the Schwarzschild metric restricted on the annulus $A_{r, 1}$, and $g_{r, 0} \equiv \delta$.
In particular,   the scalar curvature of the metric $g_{r, m}$ is zero on $A_{r, 1}$.

In the following, we will use the notation $\mathbb{S}^{n-1}(R) := \{ x \in \mathbb{R}^{n}~|~|x| = R \}$.
We can calculate the mean curvature of each boundary component (while respecting to the orientation) as follows:
\[
\begin{split}
  H_{g_{r, m}} &=
\begin{cases}
  (n-1) \left( \frac{m}{r^{n-1}} - (1 + \frac{m}{2r^{n-2}}) \right) \left( 1 + \frac{m}{2r^{n-2}} \right)^{-\frac{n}{n-2}} &~\mathrm{on}~~\mathbb{S}^{n-1}(r), \\
  (n-1) \left( 1 - \frac{m}{2} \right) \left( 1 + \frac{m}{2} \right)^{-\frac{n}{n-2}} &~\mathrm{on}~~\mathbb{S}^{n-1}(1),
\end{cases} \\
H_{\delta} &=
\begin{cases}
  - \frac{n-1}{r} &~\mathrm{on}~~\mathbb{S}^{n-1}(r), \\
  n-1 &~\mathrm{on}~~\mathbb{S}^{n-1}(1).
\end{cases}
\end{split}
\]
From these, we can compute the Yamabe energy of each metric:
\[
\begin{split}
  \mathcal{E}(g_{r, m}) &= \frac{2(n-1) \mathrm{Vol}(\mathbb{S}^{n-1}(1), \delta) \left\{ (m-r^{n-1} - \frac{m}{2}r) (1+\frac{m}{2r^{n-2}}) + 1-\frac{m^{2}}{4} \right\}}{\left( (1 + \frac{m}{2})^{\frac{2(n-1)}{n-2}} + (1 + \frac{m}{2r^{n-2}})^{\frac{2(n-1)}{n-2}} r^{n-1} \right)^{\frac{n-2}{n-1}} \mathrm{Vol}(\mathbb{S}^{n-1}(1), \delta)^{\frac{n-2}{n-1}}}, \\
\mathcal{E}(\delta|_{A_{r, 1}}) &= \frac{2(n-1)(1-r^{n-2}) \mathrm{Vol}(\mathbb{S}^{n-1}(1), \delta)}{(1+r^{n-1})^{\frac{n-2}{n-1}}\mathrm{Vol}(\mathbb{S}^{n-1}(1), \delta)^{\frac{n-2}{n-1}}}.
\end{split}
\]
From this, we can see that the Yamabe energy of both metrics are positive. 
Since $\mathcal{E}(g_{r, m})$ converges to
\[
\frac{2(n-1)(2-r-r^{n-2})\mathrm{Vol}(\mathbb{S}^{n-1}(1), \delta)}{(1+r^{n-1})^{\frac{n-2}{n-1}} \mathrm{Vol}(\mathbb{S}^{n-1}(1), \delta)^{\frac{n-2}{n-1}}}
\]
as $m \rightarrow +\infty,$ there exists $m_{0}$ such that
$$\mathcal{E}(g_{r, m}) > \mathcal{E}(\delta) > 0~~\mbox{ for all }m \ge m_{0}.$$
Therefore, from the definition of the type II Yamabe metrics, the metric $g_{r, m}$, where $m \ge m_{0}$, is not a type II Yamabe scalar-flat metric on $A_{r, 1}$.
Moreover, the  boundary $\partial A_{r, 1} = \mathbb{S}^{n-1}(r) \sqcup \mathbb{S}^{n-1}(1)$
is umbilic with respect to  $g_{r, m}$, and its mean curvature is constant.

\section{Examples of type II Yamabe metrics with non-umbilic boundary}\label{section-nonumbilic}
In this section, we give examples of metrics which are type II Yamabe with non-umbilic boundary. In particular, such examples show that the converse of the Obata-type theorem does not hold in general.

Let $(r, \theta^{1}, \cdots, \theta^{n-1})$ be the polar coordinate charts on $\mathbb{R}^{n}~(n \ge 3)$ and $\phi : \mathbb{R} \rightarrow \mathbb{R}_{+}$ a smooth nonnegative function which is compactly supported in $(\theta_{0} - \varepsilon, \theta_{0} + \varepsilon)$ for some $\theta^{1} = \theta_{0}$ and sufficiently small $\varepsilon > 0$.

\subsection{On a ball in $\mathbb{R}^{n}$ with a bump}
Consider the closed unit ball with a ``bump'': $B_{\phi} := \{ (r, \theta^{1}, \cdots, \theta^{n-1})~|~0 \le r \le \phi(\theta^{1}) + 1,~(\theta^{1}, \cdots, \theta^{n-1}) \in \mathbb{S}^{n-1}(1) \} \subset \mathbb{R}^{n}$ with the restricted Euclidean metric $\delta|_{B_{\phi}}$.
Define a function $f$ on $\mathbb{R}^{n}$ by $f(r, \theta^{1}, \cdots, \theta^{n-1}) := r - \phi(\theta^{1})$, then
$B_{\phi}$ can be rewritten as
\[
B_{\phi} = \{ (r, \theta^{1}, \cdots, \theta^{n-1})~|~f(r, \theta^{1}, \cdots, \theta^{n-1}) \le 1 \}
\]
and
\[
\mathrm{grad} f = \frac{\partial}{\partial r} - r^{-2} \dot{\phi} \frac{\partial}{\partial \theta^{1}},
\]
where $\dot{\phi} = \frac{d \phi}{d\theta^{1}}$ and $\mathrm{grad} f$ is the gradient of $f$ with respect to the Euclidean metric on $\mathbb{R}^{n}$.
Hence, the unit normal vector field $\nu$ of the boundary $\partial B_{\phi}$ is
\[
\nu = \frac{\mathrm{grad} f}{||\mathrm{grad} f||} = \frac{\frac{\partial}{\partial r} - (\phi(\theta^{1}) + 1)^{-2} \dot{\phi} \frac{\partial}{\partial \theta^{1}}}{\sqrt{1 - (\phi(\theta^{1}) + 1)^{-2} \dot{\phi}^{2}}},
\]
where $||\mathrm{grad} f||$ is the Euclidean norm of $\mathrm{grad} f$.
And for any point $p \in \partial B_{\phi}$,
\[
T_{p} \partial B_{\phi} = \left\langle \dot{\phi}(\theta^{1}(p)) \left( \frac{\partial}{\partial r} \right)_{p} + \left( \frac{\partial}{\partial \theta^{1}} \right)_{p}, \left( \frac{\partial}{\partial \theta^{2}} \right)_{p}, \cdots, \left( \frac{\partial}{\partial \theta^{n-1}} \right)_{p} \right\rangle .
\]
Then, a computation shows that we can take a function $\phi$ so that the boundary $\partial B_{\phi}$ is not totally umbilic in $B_{\phi}$.
(In particular, there is a non-umbilic point on the set $\{ (r, \theta^{1}, \cdots, \theta^{n-1}) \in \partial B_{\phi}~|~\theta^{1} = \theta_{0} \}$.)
Since $\delta|_{B_{\phi}}$ is scalar-flat, it follows that $Y_{II}(B_{\phi}, \partial B_{\phi}, [\delta|_{B_{\phi}}])$ is finite as mentioned in Section \ref{section-2}.
Moreover, from \cite{escobar1992conformal} and \cite{marques2007conformal}, $Y_{II}(B_{\phi}, \partial B_{\phi}, [\delta|_{B_{\phi}}]) < Y_{II}(\mathbb{B}^{n}, \partial \mathbb{B}^{n}, [\delta])$ since the boundary $(\partial B_{\phi}, \delta|_{B_{\phi}})$ is not totally umbilic,
and hence there is a type II Yamabe metric $g =: u^{\frac{4}{n-2}} \delta|_{B_{\phi}}$ in $[\delta|_{B_{\phi}}]$.
Since umbilicity of the boundary is invariant under conformal changes (see {\cite[Proposition 1.2]{escobar1992yamabe}}), $g$ is also not totally umbilic.

\subsection{On the unit ball in $\mathbb{R}^{n}$}
Let $\mathbb{B}^{n}$ be the unit ball (with respect to the Euclidean metric $\delta$) in $\mathbb{R}^{n}~(n \ge 3)$.
Firstly we can easily construct a family of metrics $(g_{t})_{t \in [0, \varepsilon)}$ on $\mathbb{R}^{n}$ such that
\begin{itemize}
  \item $g_{t} \equiv \delta$ on $K$ for all $t \in [0, \varepsilon)$ for a compact set $K \subset \mathbb{R}^{n}$, which contains the closure of $\mathbb{B}^{n}$,
  \item $(\partial \mathbb{B}^{n}, g_{t})$ is not umbilic in $(\mathbb{R}^{n}, g_{t})$ for all $t \in (0, \varepsilon)$,
  \item $g_{t} \rightarrow \delta$ in the smooth sense as $t \rightarrow 0$.
\end{itemize}
Then there is a small $0 < \varepsilon_{0} < \varepsilon$ such that one can obtain a positive solution $u_{t}$ of the following partial differential equation for each $t \in [0, \varepsilon_{0})$,
\[
-\frac{4(n-1)}{n-2} \Delta_{g_{t}} u_{t} + R_{g_{t}} u_{t} = 0,
\]
and $u_{t} \rightarrow 1$ at infinity.
This is done in the similar way to that in {\cite[Proof of Lemma 3.1]{lee2021geometric}}.
Hence, for each $t \in (0, \varepsilon_{0})$, the metric $\tilde{g}_{t} := u_{t}^{\frac{4}{n-2}} g_{t}$ is scalar-flat and $\partial \mathbb{B}^{n}$ is not umbilic with respect to $\tilde{g}_{t}$.
Therefore one can obtain a type II Yamabe metric $\overline{g}_{t}$ in $[\tilde{g}_{t}]$ as in the previous example.
Then $(\mathbb{B}^{n}, \overline{g}_{t}|_{\mathbb{B}^{n}})~(t \in (0, \varepsilon_{0}))$ is a type II Yamabe metric with non-umbilic boundary.

\section{CR manifolds}\label{section-cr}

Let $M$ be a compact, connected, strictly pseudoconvex
manifold of real dimension $2n+1$ equipped with the contact form $\theta$.
Let $[\theta]$ be the conformal class of $\theta$.
We consider the functional $\mathcal{E}$ given by:
$$\mathcal{E}(\theta)=\frac{\int_MR_\theta dV_\theta}{\mbox{Vol}_\theta(M)^{\frac{n}{n+1}}},$$
where $R_\theta$ is the Webster scalar curvature of $\theta$,
$dV_\theta=\theta\wedge (d\theta)^n$ is the volume form with respect to $\theta$,
and $\mbox{Vol}_\theta(M)$ is the volume of $M$ with respect to $\theta$.
We say that $\overline{\theta}$ is a \textit{CR Yamabe contact form}
if $\overline{\theta}$ is a minimizer of $\mathcal{E}|_{[\theta]}$.
Some standard terminologies and definitions 
in CR geometry and in CR Yamabe problem could be found 
in  \cite{Ho,Jerison&Lee}.

We have the following theorem, which is the corresponding 
version of Theorem \ref{thm1}
in the CR case.

\begin{theorem}\label{CR_thm}
Let $\Theta$ be a CR Yamabe contact form in $M$
with $R_\Theta>0$.
Assume that $\theta$
is a contact form in $M$ with constant Webster scalar curvature
and that $\varphi$ is a CR diffeomorphism of $M$ such that
$dV_{\varphi^*\theta}=\gamma dV_\Theta$ for some positive constant $\gamma$.
If
\begin{equation}\label{2.1}
R_\theta d\theta\leq R_\Theta d\Theta
\end{equation}
on $T^{1,0}M\times T^{0,1}(M)$, then $\theta$ is also a CR Yamabe contact form.
Moreover, if
\begin{equation}\label{2.2}
R_\theta d\theta<R_\Theta d\Theta
\end{equation}
on $T^{1,0}M\times T^{0,1}(M)$,
then $\theta$ is a unique CR Yamabe contact form (up to positive constant) in
$[\theta]$.
\end{theorem}

\begin{proof}[Proof of Theorem \ref{CR_thm}]
Since the Webster scalar curvature is preserved under the pullback action of
CR diffeomorphism, it suffices to consider the case when $\varphi=id_M$.

Let $\overline{\theta}=u^{\frac{2}{n}}\theta\in [\theta]$
where $0<u\in C^\infty(M)$.
Then
\begin{equation*}
\begin{split}
\mathcal{E}(\overline{\theta})=
\frac{\int_MR_{\overline{\theta}} dV_{\overline{\theta}}}{\mbox{Vol}_{\overline{\theta}}(M)^{\frac{n}{n+1}}}
&=\frac{\int_Mu^{-(1+\frac{2}{n})}\big(-(2+\frac{2}{n})\Delta_\theta u+R_\theta u\big)u^{2+\frac{2}{n}}dV_\theta}{(\int_Mu^{2+\frac{2}{n}}dV_\theta)^{\frac{n}{n+1}}}\\
&=\frac{\int_M\big(-(2+\frac{2}{n})u\Delta_\theta u+R_\theta u^2\big)dV_\theta}{(\int_Mu^{2+\frac{2}{n}}dV_\theta)^{\frac{n}{n+1}}},
\end{split}
\end{equation*}
where $\Delta_\theta$ is the sub-Laplacian with respect to $\theta$.
By integration by parts, we then have
$$\mathcal{E}(\overline{\theta})=
\frac{\int_M\big((2+\frac{2}{n})|\nabla_\theta u|^2+R_\theta u^2\big)dV_\theta}{(\int_Mu^{2+\frac{2}{n}}dV_\theta)^{\frac{n}{n+1}}}.$$
By the assumption (\ref{2.1}), we have
$$\frac{R_\theta}{R_\Theta}|\nabla_\Theta u|^2\leq |\nabla_\theta u|^2.$$
Combining this with the assumption that $dV_{\theta}=\gamma dV_\Theta$, we obtain
\begin{equation*}
\begin{split}
\mathcal{E}(\overline{\theta})&=
\frac{\int_M\big((2+\frac{2}{n})|\nabla_\theta u|^2+R_\theta u^2\big)\gamma dV_\Theta}{(\int_Mu^{2+\frac{2}{n}}\gamma dV_\Theta)^{\frac{n}{n+1}}}\\
&\geq\gamma^{1-\frac{n}{n+1}}\frac{R_\theta}{R_{\Theta}}\frac{\int_M\big((2+\frac{2}{n})|\nabla_\Theta u|^2+R_\Theta u^2\big)  dV_\Theta}{(\int_Mu^{2+\frac{2}{n}}  dV_\Theta)^{\frac{n}{n+1}}}\\
&=\gamma^{1-\frac{n}{n+1}}\frac{R_\theta}{R_{\Theta}}\mathcal{E}(u^{\frac{2}{n}}\Theta).
\end{split}
\end{equation*}
Since $\Theta$ is a CR Yamabe contact form  and $R_{\Theta}$, $R_\theta>0$, we have
\begin{equation}\label{2.3}
\mathcal{E}(\overline{\theta})\geq \gamma^{1-\frac{n}{n+1}}\frac{R_\theta}{R_{\Theta}}\mathcal{E}(\Theta).
\end{equation}
Note that the right hand side of (\ref{2.3}) is equal to
\begin{equation*}
\begin{split}
\gamma^{1-\frac{n}{n+1}}\frac{R_\theta}{R_{\Theta}}\mathcal{E}(\Theta)
=\gamma^{1-\frac{n}{n+1}}\frac{R_\theta}{R_{\Theta}}\frac{\int_MR_\Theta dV_\Theta}{(\int_MdV_\Theta)^{\frac{n}{n+1}}}
&=\frac{\int_MR_\theta \gamma dV_\Theta}{(\int_M\gamma dV_\Theta)^{\frac{n}{n+1}}}\\
&=\frac{\int_MR_\theta dV_\theta}{(\int_M  dV_\theta)^{\frac{n}{n+1}}}
=\mathcal{E}(\theta).
\end{split}
\end{equation*}
Therefore, by the definition of the CR Yamabe contact form,
$\theta$ is a CR Yamabe contact form.

In the above argument, if we also assume that $\overline{\theta}$
is also a CR Yamabe contact form,
then we must have $\mathcal{E}(\overline{\theta})=\mathcal{E}(\theta)$.
In particular, the inequality in (\ref{2.3}) becomes equality.
Hence, if we further assume (\ref{2.2}),
then we have $\nabla_\theta u\equiv 0$ in $M$.
Therefore, $u\equiv$ constant in $M$, for $M$ is connected.
This means that $\theta$ is a unique CR Yamabe contact form up to positive constant in $[\theta]$.
\end{proof}

\section*{Acknowledgement}
SH was supported by JSPS KAKENHI Grant Number 24KJ0153. PTH was supported by   the National Science and Technology Council (NSTC),
Taiwan, with grant Number: 112-2115-M-032 -006 -MY2.

\bibliographystyle{amsplain}

\begin{thebibliography}{30}
\bibitem{akutagawa2021obata}
K. Akuatagawa,
An Obata-type theorem on compact Einstein manifolds with boundary,
Geom. Dedicata \textbf{213} (2021), 577--587.

\bibitem{almaraz2010existence}
S. Almaraz,
An existence theorem of conformal scalar-flat metrics on manifolds with boundary,
Pacific J. Math. \textbf{248} (2010), 1--22.


\bibitem{aubin1976equations}
T. Aubin,
{\'E}quations diff{\'e}rentielles non lin{\'e}aires et probl{\`e}me de Yamabe concernant la courbure scalaire,
J. Math. Pures Appl. \textbf{55} (1976), 269--296.

\bibitem{banyaga1974volume}
A. Banyaga,
Forms-volume sur les vari\'{e}t\'{e}s \`{a} bord, Enseign. Math. \textbf{20} (1974), 127--131.

\bibitem{beckner1993sharp}
W. Beckner,
Sharp Sobolev inequalities on the sphere and the Moser--Trudinger inequality,
Ann. Math. \textbf{138} (1993), 213--242.


\bibitem{cardenas2019uniqueness}
E. C{\'a}rdenas and W. Sierra,
Uniqueness of solutions of the Yamabe problem on manifolds with boundary,
Nonlinear Analysis \textbf{187} (2019), 125--133.

\bibitem{cherrie1984problemes}
P. Cherrier,
Probl\`{e}mes de Neumann non lin\'{e}aires sur les vari\'{e}t\'{e}s riemanniennes,
J. Funct. Anal. \textbf{57} (1984), 154--206.


\bibitem{escobar1988sharp}
J. Escobar,
Sharp constant in a Sobolev trace inequality,
Indiana Math. \textbf{37} (1988), 687--698.

\bibitem{escobar1990uniqueness}
J. Escobar,
Uniqueness theorems on conformal deformation of metrics, Sobolev inequalities, and an eigenvalue estimate,
Comm. Pure Appl. Math. \textbf{43} (1990), 857--883.

\bibitem{escobar1992yamabe}
J. F. Escobar,
The Yamabe problem on manifolds with boundary,
J. Diff. Geom. \textbf{35} (1992), 21--84.

\bibitem{escobar1992conformal}
J. F. Escobar,
Conformal deformation of a Riemannian metric to a scalar flat metric with constant mean curvature on the boundary,
Ann. Math. \textbf{136} (1992), 1--50.

\bibitem{escobar1994addendum}
J. F. Escobar,
Addendum: conformal deformation of a Riemannian metric to a scalar flat metric with constant mean curvature on the boundary,
Ann. Math. \textbf{139} (1994), 749--750.

\bibitem{escobar2004conformal}
J. F. Escobar and G. Garcia,
Conformal metrics on the ball with zero scalar curvature and prescribed mean curvature on the boundary,
J. Funct. Anal. \textbf{211} (2004), 71--152.

\bibitem{girouard2017spectral}
A. Girouard and I. Polterovich,
Spectral geometry of the Steklov problem,
J. Spectr. Theory \textbf{7} (2017), 321--359.



\bibitem{hamanaka2021non}
S. Hamanaka,
Non-Einstein relative Yamabe metrics,
Kodai Math. J. \textbf{44} (2021), 265--272.

\bibitem{Ho}
P. T. Ho,  
The second CR Yamabe invariant.  
Pacific J. Math. \textbf{280} (2016), no. 2, 371–400. 

\bibitem{Jerison&Lee}
D. Jerison and J. M. Lee,  
The Yamabe problem on CR manifolds.
J. Differential Geom. \textbf{25} (1987), no. 2, 167–197. 




\bibitem{shin1994examples}
S. Kato,
Examples of non-Einstein Yamabe metrics with positive scalar curvature,
Tokyo J. Math. \textbf{17} (1994), 187--189.

\bibitem{lee2021geometric}
Dan A. Lee,
Geometric Relativity,
Graduate studies in mathematics \textbf{201}, American Mathematical Society, Providence (2019).

\bibitem{marques2005existence}
F. C. Marques,
Existence results for the Yamabe problem on manifolds with boundary,
Indianan Univ. Math. J. \textbf{54} (2005), 1599--1620.

\bibitem{marques2007conformal}
F. C. Marques,
Conformal deformations to scalar-flat metrics with constant mean curvature on the boundary,
Commun. Anal. Geom. \textbf{15} (2007), 381--405.

\bibitem{obata1962certain}
M. Obata,
Certain conditions for a Riemannian manifold to be isometric with a sphere,
J. Math. Soc. Japan \textbf{14} (1962), 333--340.

\bibitem{obata1971conjectures}
M. Obata,
The conjectures on conformal transformations of Riemannian manifolds,
Bull. Amer. Math. Soc. \textbf{77} (1971), 265--270.

\bibitem{schoen1984conformal}
R. Schoen,
Conformal deformation of a Riemannian metric to constant scalar curvature,
J. Diff. Geom. \textbf{20} (1984), 479--495.

\bibitem{trudinger1968remarks}
N. Trudinger,
Remarks concerning the conformal deformation of a Riemannian structure on compact manifolds
Ann. Scuola Norm. Sup. Pisa Cl. Sci. \textbf{22} (1968), 165--274.

\bibitem{yamabe1960deformation}
H. Yamabe,
On a deformation of Riemannian structures on compact manifolds,
Osaka Math. J. \textbf{12} (1960), 21--37.

\end{thebibliography}

\end{document}